\documentclass[11pt]{amsart}
\let\oldmarginpar\marginpar
\renewcommand\marginpar[1]{\-\oldmarginpar[\raggedleft\footnotesize #1]%
{\raggedright\footnotesize #1}}
\usepackage{amssymb,amsmath,amsthm,verbatim}
\newtheorem{theorem}{Theorem}
\newtheorem{lemma}[theorem]{Lemma}

\newtheorem{proposition}[theorem]{Proposition}

\theoremstyle{remark}

\newtheorem*{remark}{Remark}
\newtheorem*{remarks}{Remarks}

\numberwithin{theorem}{section} \numberwithin{equation}{section}

\newcommand{\ord}{\text {\rm ord}}

\newcommand{\AM}{A_M}
\newcommand{\FF}{\mathcal{F}}
\newcommand{\HH}{\mathcal{H}}

\newcommand{\R}{\mathbb{R}}

\newcommand{\Q}{\mathbb{Q}}

\newcommand{\Z}{\mathbb{Z}}
\newcommand{\N}{\mathbb{N}}

\newcommand{\SL}{{\text {\rm SL}}}
\newcommand{\Qp}{\widehat{\overline{\Q}}_p}

\newcommand{\Op}{\mathcal{O}_K}

\begin{document}
\title[]{Mock modular forms as $p$-adic modular forms}

\author{Kathrin Bringmann} 
\address{Mathematical Institute\\University of
Cologne\\ Weyertal 86-90 \\ 50931 Cologne \\Germany}
\email{kbringma@math.uni-koeln.de}
\author{Pavel Guerzhoy}
\address{Department of Mathematics\\ University of Hawaii\\ Honolulu, HI 96822-2273}
\email{pavel@math.hawaii.edu}
\author{Ben Kane}
\address{Mathematical Institute\\University of
Cologne\\ Weyertal 86-90 \\ 50931 Cologne \\Germany}
\email{bkane@math.uni-koeln.de}
\date{\today}
\thanks{The first author was partially supported by NSF grant DMS-0757907 and by the Alfried Krupp prize.  The final details of the paper were completed at the AIM workshop ``Mock modular forms in combinatorics and arithmetic geometry.''  The authors would like to thank AIM for their support and for supplying a stimulating work environment.}
\subjclass[2000] {11F33, 11F37, 11F11}
\keywords{$p$-adic modular forms, mock theta functions, mock modular forms, harmonic weak Maass forms}

\begin{abstract}
In this paper, we consider the question of correcting mock modular forms in order to obtain $p$-adic modular forms. In certain cases we show that a mock modular form $M^+$ is a $p$-adic modular form.  Furthermore, we prove that otherwise the unique correction of $M^+$ is intimately related to the shadow of $M^+$.  
\end{abstract}

\maketitle

\section{Introduction and statement of results} 

In his last letter to Hardy (see \cite{Ramanujan}, pp. 127-131), Ramanujan introduced 17 examples of functions which he called \begin{it}mock theta functions\end{it}.  He stated that these forms had properties similar to those of theta functions, but were not modular forms.  A series of authors (cf. \cite{Andrews}, \cite{AndrewsHickerson}, \cite{Choi}, \cite{Hickerson}, and \cite{Watson}) investigated these functions extensively by means of complicated combinatorial arguments, proving many of Ramanujan's claims about identities between mock theta functions and modular transformation properties.  However, the source of these properties remained a mystery until Zwegers' thesis \cite{ZwegersThesis,Zwegers} related the mock theta functions to harmonic weak Maass forms.  Using the recent development of the theory of harmonic weak Maass forms \cite{BruinierFunke}, the first author and Ono obtained an exact formula for the coefficients of one of the mock theta functions \cite{BringmannOnoInvent}.  Furthermore, this investigation has lead to an infinite family of mock theta functions related to Dyson's rank statistic on partitions \cite{BringmannOnoAnnals}.  In particular, the first author and Ono determined the shadows of these forms, which allowed them to establish that certain coefficients are coefficients of \begin{it}weakly holomorphic modular forms\end{it}, i.e., those meromorphic modular forms whose possible poles lie only at the cusps, and to determine nice congruence properties.  This new perspective on mock theta functions has lead to positivity of rank differences \cite{Bringmann,BringmannKane}, a connection between the Hurwitz class numbers and overpartition rank differences \cite{BringmannLovejoy}, and a duality relating the coefficients of mock theta functions to coefficients of weakly holomorphic modular forms \cite{FolsomOno,ZwegersDual}, among a wide variety of other applications.

The emergence of the theory of harmonic weak Maass forms in explaining the properties of the mock theta functions has lead to further investigation of what are called mock modular forms.  More formally, let $k\geq 2$ and $N>0$ be integers, $\chi$ be a Dirichlet character, and $g\in S_k(N,\chi)$ be a (normalized) newform of weight $k$, level $N$, and Nebentypus $\chi$ whose coefficients lie in a number field $K_g$.  For a harmonic weak Maass form $M\in H_{2-k}(N,\chi)$ (see Section \ref{PreliminarySection} for the relevant definition), we define as usual the antiholomorphic differential operator $\xi_{2-k}:=2iy^{2-k} \overline{\frac{\partial}{\partial \overline{z}}}$.  We let $M^+$ be the holomorphic part of $M$ and $M^-$ denote its non-holomorphic part (see Section \ref{PreliminarySection}).  Following Zagier \cite{Zagier}, we refer to $M^+$ as a \begin{it}mock modular form\end{it} and $\xi_{2-k}(M)$ as the \begin{it}shadow\end{it} of the mock modular form $M^+$.  Throughout, we let $M\in H_{2-k}(N,\chi)$ be good for $g^c(z):=\overline{g(-\overline{z})}$ (for the definition of good, see \cite{BruinierOnoRhoades} or Section \ref{PreliminarySection}).  

Since mock theta functions and mock modular forms generally have transcendental coefficients, it should be quite surprising to find out that one may consider $p$-adic properties related to these coefficients.  We will prove in this paper that mock modular forms will combine with their shadows to give $p$-adic modular forms.  Serre \cite{Serre72} initiated the study of $p$-adic modular forms in order to investigate congruences between modular forms.  His construction allowed him to define a $p$-adic topology on modular forms so that two modular forms whose Fourier coefficients are congruent to high $p$-powers will be $p$-adically close and limits of increasingly $p$-adically closer sequences would exist.  Serre also used the theory of $p$-adic modular forms to give a $p$-adic interpolation of the values at negative integers of $L$-series of a totally real number field by using a convergent sequences of Eisenstein series and determining the constant term of the resulting $p$-adic modular form.  This groundwork has lead to a number of applications.  For example, Ahlgren and Boylan \cite{AhlgrenBoylan} obtained deep congruences for the partition function and other modular forms.

We now turn to the statement of our results.  For this, let $p$ be a prime and fix an algebraic closure $\overline{\Q}_p$ of $\Q_p$ along with an embedding $\iota: \ \overline{\Q} \hookrightarrow \overline{\Q}_p$.  We let $\Qp$ denote the $p$-adic closure of $\overline{\Q}_p$ and normalize the $p$-adic order so that $\ord_p(p)=1$.  We do not distinguish between algebraic numbers and their images under $\iota$.  In particular, for algebraic numbers $a,b \in \overline{\Q}$ we write $a \equiv b \pmod{p^m}$ if $\ord_p(\iota(a-b)) \geq m$.  For a formal power series $H(q)=\sum_{n\in \Z} a(n) q^n \in \Qp[[q,q^{-1}]]$, we write $H \equiv 0 \pmod {p^m}$ if $\displaystyle \sup_{n\in \Z} \left(\ord_p(a(n))\right) \geq m$.

In this paper, a \begin{it}$p$-adic modular form\end{it} of level $N$ and Nebentypus $\chi$ will refer to a formal power series $H(q)={\displaystyle \sum_{n\geq -t}} a(n) q^n$ with coefficients in $\Qp$ satisfying the following condition:  for every $m\in \N$ there exists a weakly holomorphic modular form ($q:=e^{2\pi i z}$ for $z\in \mathbb{H}$) $H_m(z)={\displaystyle \sum_{n\geq - t}} b_m(n) q^n\in M_{\ell_m}^!(N,\chi)$ (the space of weakly holomorphic modular forms of weight $\ell_m$, level $N$, and Nebentypus $\chi$), with algebraic coefficients $b_m(n)\in \overline{\Q}$, which satisfies the congruence
\begin{equation}
H\equiv H_m\pmod{p^m}.
\end{equation}
If $\ell_m=\ell$ is constant, then we will refer to $H$ as a $p$-adic modular form of weight $\ell$, level $N$, and Nebentypus $\chi$.

\begin{remark}
One can relate $p$-adic modular forms in our sense to those in the sense of Serre.  For this, let $\Delta$ be the unique normalized newform of weight $12$ for $\SL_2(\Z)$.  Consider the case of trivial Nebentypus and level $N=p^s$ ($s\in \N_0$).  After multiplication by $\Delta^t$, where $t$ is the appropriate power as given in the expansion of $H(q)$ above, every $p$-adic modular form in our sense becomes a $p$-adic modular form of weight $\ell$ in the sense of Serre due to Theorem 5.4 of \cite{Serre}.  
\end{remark}

In this paper, we will determine the unique correction (up to addition by a $p$-adic modular form) needed to complete $M^+$ to a $p$-adic modular form.  We will prove that when $M^+$ is not itself a $p$-adic modular form, this series has an intimate relationship with the shadow of $M^+$.  To determine the correction needed, we will first establish some notation.  We define the \begin{it}Eichler integral\end{it} of $g$ by
\begin{equation}\label{Egdefeqn}
E_g(z):=\sum_{n\geq 1} n^{1-k} a_g(n) q^n,
\end{equation}
where $a_g(n)$ denotes the $n$-th coefficient of $g$.  One easily sees that $D^{k-1}(E_g)=g$, where $D:=q\frac{\partial}{\partial q}$.  In Theorem 1.1 of \cite{PavelKentOno}, the second author, Kent, and Ono have shown that
$$
\FF_{a_M(1)}:=M^+ - a_M(1) E_g
$$
has coefficients in $K_g$, where we have abused notation to denote the $n$-th coefficient of $M^+$ by $a_M(n)$.  The embedding $\iota$ allows us to consider $\FF_{a_M(1)}$ as an element of $\Qp[[q,q^{-1}]]$.  Hence for every $\gamma\in \Qp$, 
\begin{equation}\label{FFadefeqn}
\FF_{\alpha}:=M^+ - \alpha E_g:= \FF_{a_M(1)}- \gamma E_g
\end{equation}
is an element of $\Qp[[q,q^{-1}]]$.  Here $\alpha:=a_M(1)+\gamma\in \AM$, where
$$
\AM:=a_M(1)+\Qp:=\left\{ a_M(1)+x:x\in \Qp\right\}
$$
denotes a set of formal sums.

As an example of the algebraicity of the coefficients of $\FF_{\alpha}$, consider the case where $g=\Delta$ and $M^+$ is the mock modular form associated to $\Delta$ with principal part $q^{-1}$, so that $11!M^+$ is approximately given by
$$
11!q^{-1} -\frac{2615348736000}{691} -73562460235.683647469 q  -929026615019.113082 q^2 +\cdots.
$$
Then a computer calculation indicates that
$$
11! \FF_{a_M(1)} = 11!q^{-1} -\frac{2615348736000}{691} -929888675100 q^2 - \frac{80840909811200}{9} q^3 - \cdots.
$$
This numerically demonstrates the above statement that all of the coefficients of $\FF_{a_M(1)}$ are rational, as $K_{\Delta}=\Q$.

The fact that the function $\FF_{\alpha}$ defined in \eqref{FFadefeqn} has coefficients in $\Qp$ now permits one to consider its $p$-adic properties.  We begin with the case that $p\nmid N$.  Let $\beta,\beta'$ be the roots of the polynomial
$$
x^2-a_g(p) x +\chi(p) p^{k-1} = (x-\beta)(x-\beta'),
$$
ordered so that $\ord_p(\beta)\leq \ord_p(\beta')$.  We first treat the ``generic'' case when $\ord_p(\beta)< \ord_p(\beta')$.  For this, we consider
\begin{equation}\label{FF*adefeqn}
\FF_{\alpha}^*:=\FF_{\alpha}-p^{1-k}\beta' \FF_{\alpha}|V(p)=\left(M^+-\alpha E_g\right)|\left( 1- p^{1-k}\beta' V(p)\right),
\end{equation}
where $V(p)$ is the usual $V$-operator.  

For instance, in the above example where $g=\Delta$, denote the coefficients of $11!\FF_{a_M(1)}$ (resp. $11!\FF_{a_M(1)}^*$) by $c(n)\in \Q$ (resp. $c^*(n)$).  Writing these coefficients $3$-adically, a computer calculation indicates that
\begin{eqnarray*}
c^*(3)=c(3)&= &3^{-2}+3^{-1}+2+\cdots,\\
c\left(3^6\right)&=& 3^{-47} + 3^{-46} + 2 \left(3^{-45}\right) + \cdots,\\
c^*\left(3^6\right)&=& 3^{-47} + 3^{-46}+2\left(3^{-45}\right) + \cdots, \\
c\left(3^7\right)&=& 3^{-56} + 3^{-55} + 2\left(3^{-54}\right) + \cdots,\\
c^*\left(3^7\right)&=& 3^{-56}+ 3^{-55} + 2\left(3^{-54}\right) + \cdots. 
\end{eqnarray*}
Having dispatched the difficulty arising from transcendental coefficients, our goal of realizing $M^+$ as a $3$-adic modular form now seems to be thwarted by the high powers of $3$ appearing in the denominators.  However, choosing $\alpha=a_M(1)+\gamma/11!$ with 
$$
\gamma:=154300462955809413372268553898=3^7+3^8+\cdots
$$
and denoting the coefficients of $11!\FF_{\alpha}$ (resp. $11!\FF_{\alpha}^*$) by $c_{\alpha}(n)$ (resp. $c_{\alpha}^*(n)$), one sees that 
\begin{eqnarray*}
c_{\alpha}^*(3)=c_{\alpha}(3)&= &2\left(3^5\right) + 3^7 + 3^8 + \cdots,\\
c_{\alpha}\left(3^6\right)&=& 2\left(3^{-5}\right) + 2\left(3^{-3}\right) + 3^{-2}+\cdots,\\
c_{\alpha}^*\left(3^6\right)&=&3^5+3^6+3^8 + \cdots,\\
c_{\alpha}\left(3^7\right)&=& 2\left(3^{-7}\right) + 2\left(3^{-5}\right) + 2\left(3^{-4}\right) + \cdots,\\
c_{\alpha}^*\left(3^7\right)&=& 2+2(3) + 2(3^2) + \cdots.
\end{eqnarray*}
Hence the denominators of the coefficients of $\FF_{\alpha}^*$ seem to have vanished, while the denominators of $\FF_{\alpha}$ still seem to grow.  

In order to cancel the denominators of $\FF_{\alpha}$ itself, we will require a further refinement of our corrected series.  For this purpose, for $\delta\in \Qp$, we define
\begin{equation}\label{FFaddefeqn}
\FF_{\alpha,\delta}:=\FF_{\alpha}-\delta\left(E_{g}-\beta E_{g|V(p)}\right).
\end{equation}
Denoting the coefficients of $11!\FF_{\alpha,\delta}$ by $c_{\alpha,\delta}(n)$, we see in our above example (with $\alpha$ chosen as above and $\delta$ chosen to be $23974292034=2\left(3^7\right)+2\left(3^9\right) + 3^{10}+\cdots$), that 
\begin{eqnarray*}
c_{\alpha,\delta}\left(3^6\right)&=& 3^5+3^6+3^8 + \cdots,\\
c_{\alpha,\delta}\left(3^7\right)&=& 2+2(3) + 2(3^2) + \cdots.
\end{eqnarray*}
The phenomenon we have observed here is evidenced by the fact that the correct choice of $\alpha$ and $\delta$ will result in a $p$-adic modular form. 
 The choice of $\alpha$ will be given in terms of the coefficients of $D^{k-1}\left(\FF_{a_{M}(1)}\right)$ and $D^{k-1}\left(\FF_{a_{M}(1)}^*\right)$.  For ease of notation, we define $b_M(n),b_M^*(n)\in \Qp$ by 
\begin{eqnarray}
\label{bMdefeqn}
F_{a_M(1)}:=D^{k-1}\left(\FF_{a_M(1)}\right)&=& \sum_{n\gg -\infty} b_M(n) q^n,\\
\label{bM*defeqn}
F_{a_M(1)}^*:=D^{k-1}\left(\FF_{a_M(1)}^*\right)&=& \sum_{n\gg -\infty} b_M^*(n) q^n.
\end{eqnarray}
\begin{theorem}\label{genericthm}
\noindent
\begin{enumerate}
\item[] \hspace{-36pt}\textnormal{(1)}\hspace{21pt}
There exists exactly one 
$\alpha\in \AM$ such that $\FF_{\alpha}^*$ is a $p$-adic modular form of weight $2-k$, level $pN$, and Nebentypus $\chi$.  The unique choice of $\alpha$ is given by the $p$-adic limit
\begin{equation}\label{alpha*evalgeneqn}
\alpha = a_M(1) + \lim_{m\to\infty} \frac{b_M^*\left(p^{m}\right)}{\beta^{m}}.
\end{equation}
\item[] \hspace{-36pt}\textnormal{(2)}\hspace{21pt}
If $\ord_p(\beta)<\ord_p(\beta')\neq (k-1)$, then there exists exactly one pair $(\alpha,\delta)\in \AM \times \Qp$ such that $\FF_{\alpha,\delta}$ is a $p$-adic modular form of weight $2-k$, level $pN$, and Nebentypus $\chi$.   Moreover, $\alpha$ is the unique choice from part (1) and 
\begin{equation}\label{deltaevaleqn}
\delta = \lim_{m\to\infty} \frac{a_{\FF_{\alpha}}\left(p^m\right)p^{m(k-1)}}{{\beta'}^{m}}.
\end{equation}
\end{enumerate}
\end{theorem}
\begin{remarks}
\noindent
\begin{enumerate}
\item[]\hspace{-36pt}(1)\hspace{21pt}Whenever $\ord_p(\beta)<\ord_p(\beta')$, we will furthermore show that equation \eqref{alpha*evalgeneqn} may be rewritten  as
\begin{equation}\label{alphaevalgeneqn}
\alpha = a_M(1) + \left(\beta-\beta'\right) \lim_{m\to\infty} \frac{b_M\left(p^{m}\right)}{\beta^{m+1}}.
\end{equation}
\item[] \hspace{-36pt}(2)\hspace{21pt}One predicts that $\alpha$ is transcendental whenever $g$ does not have CM, as conjectured by the second author, Kent, and Ono \cite{PavelKentOno}, and hence the dependence on $a_M(1)$ in equation \eqref{alphaevalgeneqn} is likely unavoidable in the general case.
\item[] \hspace{-36pt}(3)\hspace{21pt}The limit occurring in equation \eqref{alphaevalgeneqn} also occurs (up to multiplication by a constant in $K_g(\beta)$) as the first coefficient of a limit in Proposition 2.2 of \cite{PavelKentOno}.  In \cite{PavelKentOno}, the authors conclude that this series has coefficients in $\overline{\Q}_p$.  However, due to the fact that the algebraic closure $\overline{\Q}_p$ of $\Q_p$ is not $p$-adically closed, one cannot actually conclude that this limit is an element of $\overline{\Q}_p$, but only necessarily an element of the larger field $\Qp$.
\item[]\hspace{-36pt}(4)\hspace{21pt}As Zagier observed in \cite{Zagier}, Eichler integrals are mock modular forms when one extends the definition to include shadows which are weakly holomorphic modular forms.  Hence our $p$-adic modular forms $\FF_{\alpha}^*$ and $\FF_{\alpha,\delta}$ transform like mock modular forms.
\end{enumerate}
\end{remarks}
Consider again the example of the mock modular form $M^+$ associated to $g=\Delta$ with principal part $q^{-1}$ which we have computationally investigated above.  Since $g$ has trivial level and the principal part is $q^{-1}$, taking the unique choice of $\alpha\in \AM$ and $\delta\in \Qp$ given by Theorem \ref{genericthm} (2), $\Delta\FF_{\alpha,\delta}$ will be a $3$-adic modular form in the sense of Serre.  Quite pleasantly, computational evidence indicates that
\begin{eqnarray*}
\Delta\FF_{\alpha,\delta} &\equiv& 1\pmod{3},\\
\Delta\FF_{\alpha,\delta} &\equiv& E_2\pmod{3^2},\\
\Delta\FF_{\alpha,\delta} &\equiv& E_2+9\Delta\pmod{3^3},\\
&\vdots&
\end{eqnarray*}
where $E_2(z):=1-24\sum_{n=1}^{\infty} \sigma_1(n) q^n,$ with $\sigma_r(n):=\sum_{d\mid n} d^r.$

Theorem \ref{genericthm} gives a strong interplay between $g$ and the weight $2-k$ correction of 
$$
\FF_0^*=M^{+}-p^{1-k}\beta'M^{+}|V(p)
$$
in the case when $\alpha\neq 0$.  It is of course of particular interest to consider the case when $\alpha=0$, so that $\FF_0^*$ is itself a $p$-adic modular form.  This leads us immediately to consider the case when $g$ has CM by an imaginary quadratic field $K$.
\begin{theorem}\label{splitthm}
Assume that $g$ has CM by $K$, $p\nmid N$, and $p$ is split in $\Op$.  Then $\alpha=0$ is the unique choice from Theorem \ref{genericthm}, so that $\FF_0^*$ is a $p$-adic modular form of weight $2-k$, level $pN$, and Nebentypus $\chi$.
\end{theorem}

When we are not in the generic case, the situation can differ greatly, as the next theorem reveals.  In the case where $g$ has CM and $p$ is inert, the extra symmetry $\beta'=-\beta$ allows us to obtain a $p$-adic modular form correction of $M^+$ itself.  To do so, for $\alpha\in \Qp$, we define another type of corrected series
\begin{equation}\label{FFtadefeqn}
\widetilde{\FF}_{\alpha}:=M^+ - \alpha E_{g|V(p)}.
\end{equation}

\begin{theorem}\label{inertthm}
Assume that $g$ has CM by $K$, $p$ is inert in $\Op$, and $p\nmid N$.  Then there exists precisely one $\alpha\in \Qp$ such that $\widetilde{\FF}_{\alpha}$ is a $p$-adic modular form of weight $2-k$, level $pN$, and Nebentypus $\chi$, given by the $p$-adic limit
\begin{equation}\label{alphaevalinerteqn}
\alpha=\lim_{m\to\infty} \frac{a_{D^{k-1}(M^+)} \left(p^{2m+1}\right)}{\beta^{2m}}.
\end{equation}
\end{theorem}
\begin{remarks}
\noindent
\begin{enumerate} 
\item[]\hspace{-36pt}(1)\hspace{21pt}It is worth remarking that the proof of Theorem \ref{inertthm} shows more generally that there exists exactly one pair $(\alpha,\delta)\in \AM\times \Qp$ such that $\FF_{\alpha,\delta}$ is a $p$-adic modular form whenever $p\nmid N$ and $a_g(p)=0$, but the properties of the CM form allow us to relate $\delta$ and $\alpha$ in this case. 
\item[]\hspace{-36pt}(2)\hspace{21pt}It is an interesting question whether $\alpha=0$ ever occurs in Theorem \ref{inertthm}.  In a particular example considered in \cite{PavelKentOno}, it was shown that this does not happen for every prime  $p<32500$.
\end{enumerate}
\end{remarks}

For every $\alpha\in \AM$, we define
$$
\widetilde{F}_{\alpha}:=D^{k-1}\left(M^+ - \alpha E_{g|V(p)}\right),
$$
so that if $g$ has CM by $K$ and $p$ is inert in $\Op$, then for $\alpha$ given in equation \eqref{alphaevalinerteqn}, $\widetilde{F}_{\alpha}$ is the image under $D^{k-1}$ of the $p$-adic modular form given in Theorem \ref{inertthm}.  Let $U(p)$ denote the usual $U$-operator.  We will conclude Theorem \ref{inertthm} from the following proposition.  
\begin{proposition}\label{inertprop}
Assume that $g$ has CM by $K$, $p$ is inert in $\Op$, and $p\nmid N$.  
Then for all but exactly one $\alpha\in \Qp$, we have the $p$-adic limit
\begin{equation}\label{inertpropeqn}
\lim_{m\to\infty} \frac{\widetilde{F}_{\alpha}|U\left(p^{2m+1}\right)}{a_{\widetilde{F}_{\alpha}}\left(p^{2m+1}\right)}=g.
\end{equation}
\end{proposition}
\begin{remark}
Proposition \ref{inertprop} corrects the original statement of Theorem 1.2 (2) of \cite{PavelKentOno}, which contained an error that was pointed out and fixed by the third author.  The proof given here was then subsequently included in the final version of \cite{PavelKentOno}.  The existence of the limit in the original statement, with 
\begin{equation}\label{Fadefeqn}
F_{\alpha}:=D^{k-1}(\FF_{\alpha})
\end{equation}
instead of $\widetilde{F}_{\alpha}$, is independent of $\alpha$, and one cannot even conclude existence of that limit without determining explicitly that $\alpha\neq 0$ in this theorem.  Even if the limit with $F_{\alpha}$ would exist, such a limit would have a linear dependence on $\alpha$ times $g|V(p)$, and hence could not equal $g$, as stated in \cite{PavelKentOno}.
\end{remark}

Lastly, we deal with the ``bad'' primes.
\begin{theorem}\label{divisorsNthm}
Let $p\mid N$ and $g$ be a newform.  Then the following hold.
\noindent
\begin{enumerate}
\item[]\hspace{-36pt}\textnormal{(1)}\hspace{21pt}If $a_g(p)=0$, then $\FF_{\alpha}$ is a $p$-adic modular form of level $N$, and Nebentypus $\chi$ for every $\alpha\in \AM$.
\item[]\hspace{-36pt}\textnormal{(2)}\hspace{21pt}If $a_g(p)\neq 0$, then there exists exactly one $\alpha\in \AM$ such that $\FF_{\alpha}$ is a $p$-adic modular form of weight $2-k$, level $N$, and Nebentypus $\chi$.  The unique choice of $\alpha$ is given by 
\begin{equation}\label{alphadivNeqn}
\alpha=a_M(1) + \lim_{m\to\infty} \frac{b_M\left(p^m\right)}{a_g(p)^m},
\end{equation}
where $b_M\left(n\right)$ is given in equation \eqref{bMdefeqn}.
\end{enumerate}
\end{theorem}
\begin{remark}
If $g$ has CM, then $a_M(1)\in \Qp$ and hence in this case $\AM=\Qp$.  In particular, the choice $\alpha=0\in \AM$ is valid.  Furthermore, when $p\mid N$, the condition $a_g(p)=0$ appearing in Theorem \ref{divisorsNthm} (1) is equivalent to the restriction $p^2\mid N$ (cf. \cite{OnoBook}, p. 29) and hence depends only on the level and not the given form $g$.  Therefore when $g$ has CM and $p^2\mid N$, then $M^+$ and $E_g$ are both $p$-adic modular forms of level $N$ and Nebentypus $\chi$.  One also concludes from Theorem \ref{divisorsNthm} (2) that if $g$ has CM and $p|| N$, then there exists exactly one $\alpha\in \Qp$ such that $\FF_{\alpha}$ is a $p$-adic modular form of weight $2-k$, level $N$, and Nebentypus $\chi$.  Specifically, for those primes $p$ which divide the discriminant of the CM field but do not divide the norm of the conductor of the Gr\"ossencharacter associated to $g$, there is precisely one $\alpha\in\Qp$ such that $\FF_{\alpha}$ is a $p$-adic modular form.
\end{remark}

The paper is organized as follows:  Section \ref{PreliminarySection} is devoted to a quick overview of the relevant definitions involving harmonic weak Maass forms.  In Section \ref{GenericSection}, we define an operator which recasts the question about our corrected series being a $p$-adic modular form to one about a certain $p$-adic limit equaling zero.  We then use this equivalence to establish the results in the generic case.  In Section \ref{CMSection} we combine the aforementioned equivalence and the fact that a certain quadratic twist is a weakly holomorphic modular form to establish the results when $g$ has CM.

\section{Basic facts on harmonic weak Maass forms}\label{PreliminarySection}

We begin by recalling the relevant definitions and some important facts about harmonic weak Maass forms.  For further details, we refer the reader to \cite{BruinierFunke}.  We will write an element of the upper half plane as $z=x+iy$ with $x,y\in \R$.  Denote the \begin{it}weight $2-k$ hyperbolic Laplacian\end{it} by 
$$
\Delta_{2-k}:=-y^2\left(\frac{\partial^2}{\partial x^2}+\frac{\partial^2}{\partial y^2}\right)+i(2-k) y\left(\frac{\partial}{\partial x} + i\frac{\partial}{\partial y}\right).
$$
A \begin{it}harmonic weak Maass form\end{it} $M$ of weight $2-k$, level $N$, and Nebentypus $\chi$ is defined as a smooth function on the upper half plane which satisfies the following properties:
\begin{enumerate}
\item 
For all $\left(\begin{smallmatrix} a&b \\ c &d\end{smallmatrix}\right)\in \Gamma_0(N)$ we have
$$
M\left(\frac{az+b}{cz+d}\right) = \chi(d) (cz+d)^{2-k}M(z).
$$ 
\item We have that $\Delta_{2-k}(M)=0$.
\item The function $M(z)$ has at most linear exponential growth at all cusps of $\Gamma_0(N)$.
\end{enumerate}
We denote this space of harmonic weak Maass forms by $H_{2-k}(N,\chi)$. 

We now restrict our attention to the subspace of harmonic weak Maass forms which map to cusp forms under the $\xi$-operator.  Taking the Fourier expansion in terms of $e^{2\pi i x}$, $M$ decomposes naturally as 
$$
M(z)=M^+(z) +M^-(z),
$$
where the \begin{it}holomorphic part\end{it} of $M$ is given by
$$
M^+(z):=\sum_{n\gg -\infty} a_{M}(n) q^n
$$
and the \begin{it}non-holomorphic part\end{it} of $M$ is defined by
\begin{equation}\label{M-eqn}
M^-(z):= \sum_{n>0} a_{M}^{-}(-n) \Gamma\left(1-k;4\pi n y\right) q^{-n}.
\end{equation}
Here
$$
\Gamma(a;x):=\int_x^{\infty} e^{-t}t^{a-1} dt
$$
is the incomplete $\Gamma$-function.  One refers to 
$$
\sum_{n\leq 0} a_{M}(n)q^n
$$
as the \begin{it}principal part\end{it} of $M$.

As defined in \cite{BruinierOnoRhoades}, we say that a harmonic weak Maass form $M\in H_{2-k}(N,\chi)$ is \begin{it}good for\end{it} the cusp form $g\in S_k(N,\overline{\chi})$ if the following conditions are satisfied:
\begin{itemize}
\item[(i)] The principal part of $M$ at the cusp $\infty$ belongs to $K_g[q^{-1}]$.
\item[(ii)] The principal parts of $M$ at other cusps of $\Gamma_0(N)$ are constant.
\item[(iii)] We have that $\xi_{2-k}(M)=\|g\|^{-2}g$, where $\| \cdot \|$ is the usual Petersson norm.
\end{itemize}

We let $T(n)$ denote the $n$-th Hecke operator.  The action of the Hecke operator on a harmonic weak Maass form $M$ is defined analogously to that on holomorphic modular forms.  We denote this action by $M|_{2-k} T(n)$, suppressing the weight in the notation when it is clear from the context.  We extend this definition to formal power series in the obvious way.  One can show that (see \cite{BruinierOnoRhoades}, the proof of Theorem 1.3)
\begin{equation}\label{Heckeeqn}
M|_{2-k}T(n) = n^{1-k} a_g(n) M +R_n,
\end{equation}
for some weakly holomorphic modular form $R_n\in M_{2-k}^!(N,\chi)$.  We use the commutation relation between $D^{k-1}$ and $T(n)$ given by
\begin{equation}\label{Heckecommuteeqn}
n^{k-1}D^{k-1}\big( E|_{2-k}T(n)\big) = \left(D^{k-1} E\right)|_k T(n),
\end{equation}
which is valid for every formal power series $E$.  This yields that 
\begin{equation}\label{rneqn}
F_{\alpha}|_k T(n) = a_g(n) F_{\alpha} + r_n,
\end{equation}
where $F_{\alpha}$ was defined in \eqref{Fadefeqn} and 
\begin{equation}\label{rndefeqn}
r_n:=n^{k-1} D^{k-1}(R_n).
\end{equation}
Similarly to $F_{\alpha}$, for $\FF_{\alpha,\delta}$ defined in \eqref{FFaddefeqn}, we let
\begin{equation}\label{Faddefeqn}
F_{\alpha,\delta}:=D^{k-1}(\FF_{\alpha,\delta}).
\end{equation}
When $n$ and $p$ are coprime, one obtains
\begin{equation}\label{rnadeqn}
F_{\alpha,\delta}|_k T(n) = a_g(n) F_{\alpha,\delta} + r_n.
\end{equation}

\section{A Criterion for $p$-adic modular forms and the proof of Theorem \ref{genericthm}}\label{GenericSection}

The uniqueness of $\alpha$ in the theorems will follow from the fact that $E_{g}$ and $E_{g|V(p)}$ are not $p$-adic modular forms when $p\nmid N$.  We shall show this through the following simple lemma.

\begin{lemma}\label{uniquelemma}
\noindent
\begin{enumerate}
\item Assume that $p\nmid N$ and either $\ord_p(\beta)<\ord_p(\beta')$ or $\beta'=-\beta$.  Then $E_g$, $E_{g|V(p)}$, and $E_g-\beta'p^{1-k}E_{g}|V(p)=E_{g-\beta' g|V(p)}$ are not $p$-adic modular forms.  Furthermore, if $\ord_p{\beta'}\neq k-1$, then $E_{g-\beta g|V(p)}$ is not a $p$-adic modular form.  If $p || N$ then $E_g$ is not a $p$-adic modular form.
\end{enumerate}
\end{lemma}
\begin{proof}
Let $H$ be a $p$-adic modular form.  We first note that there exists a constant $A\geq 0$ such that $\ord_p\left(a_H(n)\right)\geq -A$ for every $n\in\Z$.  This follows easily from the fact that the coefficients of weakly holomorphic modular forms have bounded denominators, which can be concluded from the corresponding statement for holomorphic modular forms in Theorem 3.52 of \cite{Shimura}.  Hence for every $m\in \N$ we have
$$
\ord_p\left(a_{D^{k-1}(H)}\left(p^{m}\right)\right)\geq m(k-1)-A.
$$

Let $C\geq 0$ be given.  Using the fact that $g$ is a Hecke eigenform, the statements for $E_g$ and $E_{g|V(p)}$ follow by 
$$
a_{g|V(p)}\left(p^{2m+1}\right)=a_{g}\left(p^{2m}\right)=\sum_{\ell=0}^{2m} \beta^{\ell}{\beta'}^{2m-\ell}=\beta^{2m} \sum_{\ell=0}^{2m} \left(\frac{\beta'}{\beta}\right)^{\ell},
$$
which, since $\ord_p(\beta)<\ord_p(\beta')$ or $\beta'=-\beta$, has $p$-order $\ord_p\left(\beta^{2m}\right)\leq \frac{2m(k-1)}{2}<2m(k-1)-C$ for $m$ sufficiently large.

The equality $E_g - \beta' p^{1-k} E_{g}|V(p)= E_{g-\beta'g|V(p)}$ follows by the commutation relation 
\begin{equation}\label{VDeqn}
\left( D^{k-1}(f)\right)|V(p) =  p^{1-k} D^{k-1}\left(f|V(p)\right)
\end{equation}
between $D^{k-1}$ and the $V$-operator on a formal power series $f(q)=\sum_{n\gg -\infty} a_{f}(n)q^n$.  The lemma for $E_{g-\beta'g|V(p)}$ then follows by the fact that for $m\geq 0$ one has $a_{g-\beta'g|V(p)}\left(p^m\right)=\beta^m$ , which has $p$-order at most $\frac{m(k-1)}{2}<m(k-1)-C$.  One concludes the lemma for $E_{g-\beta g|V(p)}$ similarly by the fact that $a_{g-\beta g|V(p)}\left(p^m\right)={\beta'}^m$ and the assumption that $\ord_p\left(\beta'\right)<k-1$.  

Finally, when $p || N$ the $p^m$-th coefficient of $E_g$ is $a_g(p)^{m} p^{m(1-k)}$.  Since $\ord_p\left(a_g(p)\right)=\frac{k}{2}-1$ one again concludes the lemma in this case.
\end{proof}

In order to deduce a helpful equivalence relation which will determine whether our corrected series are $p$-adic modular forms, one defines the operator
\begin{equation}\label{Beqn}
B(p):=-\beta p^{1-k}\left(1- \beta' p^{1-k} V(p)\right)\left( 1-\beta^{-1}p^{k-1}U(p)\right).
\end{equation}
In order to determine the action of $B(p)$ on our corrected series $\FF_{\alpha}$ and $\FF_{\alpha,\delta}$, we will first determine the actions of $B(p)$ and $B(p)U(p)$ on the correction terms $E_g$ and $E_{g}-\beta'E_{g|V(p)}$.  This is established in the following lemma.
\begin{lemma}\label{Bannlemma}
The series $E_g$ is annihilated by $B(p)$.  Furthermore, $E_{g|V(p)}$ is annihilated by $B(p)U(p)$.
\end{lemma}
\begin{proof}
Since $V(p)U(p)$ acts like the identity, $\beta+\beta'=a_g(p)$ and $\beta\beta'=\chi(p)p^{k-1}$, one obtains that
\begin{equation}\label{Bprewriteeqn}
B(p)=U(p)+\chi(p)p^{1-k}V(p)-a_g(p) p^{1-k}.
\end{equation}
However, rewriting the right hand side in terms of the weight $2-k$ Hecke operator $T(p)$ and using the commutation relation \eqref{Heckecommuteeqn} between $T(p)$ and $D^{k-1}$ yields that
$$
D^{k-1}\left(E_g|B(p)\right) = p^{1-k} g | \big(T(p)-a_g(p)\big)=0.
$$
One concludes that $E_g|B(p)=0$, giving the first statement of the lemma.  

Since we have already seen that $E_g$ is annihilated by $B(p)$, the second statement is equivalent to the statement that $E_{g-\beta'g|V(p)}$ is annihilated by $B(p)U(p)$.  Now observe that 
\begin{equation} \label{BpUpeqn}
B(p)U(p)=\beta\beta' p^{2(1-k)}\left(1-{\beta'}^{-1}p^{k-1}U(p)\right)\left( 1-\beta^{-1}p^{k-1}U(p)\right).
\end{equation}
Thus $B(p)U(p)$ clearly annihilates any formal power series which is an eigenform under the $U(p)$-operator with eigenvalue $\beta p^{1-k}$ or $\beta'p^{1-k}$.  To finish the proof, we use the commutation relation 
\begin{equation}\label{UDeqn}
\left( D^{k-1}(f)\right)|U(p) = p^{k-1} D^{k-1}\left(f|U(p)\right)
\end{equation}
between $D^{k-1}$ and the $U$-operator, on a formal power series $f(q)=\sum_{n\gg -\infty} a_{f}(n)q^n$.  Together with the fact that $g-\beta'g|V(p)$ is an eigenform under the $U(p)$-operator with eigenvalue $\beta$, this yields 
$$
E_{g-\beta'g|V(p)}|U(p)=p^{1-k}E_{\left(g-\beta'g|V(p)\right)|U(p)} = \beta p^{1-k}E_{g-\beta'g|V(p)},
$$
and the claim follows.
\end{proof}
Recall the definitions of $\FF_{\alpha}$, $F_{\alpha}$, $\FF_{\alpha}^*$, $\FF_{\alpha,\delta}$, and $F_{\alpha,\delta}$ given in equations \eqref{FFadefeqn}, \eqref{Fadefeqn}, \eqref{FF*adefeqn}, \eqref{FFaddefeqn}, and \eqref{Faddefeqn}, respectively.  Furthermore, for $\alpha\in \AM$ and $\delta\in \Qp$, define
\begin{eqnarray}\label{FF*addefeqn}
\FF_{\alpha,\delta}^*&:=&\FF_{\alpha,\delta}| \left(1-\beta' p^{1-k} V(p)\right),\\
\label{F*addefeqn} F_{\alpha,\delta}^*&:=&D^{k-1}(\FF_{\alpha,\delta}^*)= F_{\alpha,\delta}|\left(1-\beta'V(p)\right),
\end{eqnarray}
which follows by the commutation relation \eqref{VDeqn}.  We will bootstrap from the following proposition in order to obtain Theorem \ref{genericthm}.  
\begin{proposition}\label{padiclimprop}
Let $\alpha\in \AM$ and $\delta\in \Qp$.  
\noindent
\begin{enumerate}
\item[]\hspace{-36pt}\textnormal{(1)}\hspace{21pt}The form $\FF_{\alpha,\delta}^*|U(p)$ is a $p$-adic modular form of weight $2-k$, level $pN$, and Nebentypus $\chi$ if and only if 
\begin{equation}\label{had*defeqn}
h_{\alpha,\delta}^{*}:=\lim_{m\to \infty} \left( \beta^{-m} F_{\alpha,\delta}^*|U\left(p^m\right)\right) =0.
\end{equation}
\item[]\hspace{-36pt}\textnormal{(2)}\hspace{21pt}If $\ord_p(\beta)<\ord_p(\beta')$, then $\FF_{\alpha}^*$ is a $p$-adic modular form of weight $2-k$, level $pN$, and Nebentypus $\chi$ if and only if 
\begin{equation}\label{hadefeqn}
h_{\alpha}:=\lim_{m\to \infty}\left(\beta^{-m} F_{\alpha}|U\left(p^m\right) \right) =0.
\end{equation}
\item[]\hspace{-36pt}\textnormal{(3)}\hspace{21pt}Finally, if $\ord_p(\beta')\neq k-1$, then $\FF_{\alpha,\delta}$ is a $p$-adic modular form of weight $2-k$, level $pN$, and Nebentypus $\chi$ if and only if 
\begin{equation}\label{haddefeqn}
h_{\alpha,\delta}:=\lim_{m\to\infty}\left( {\beta'}^{-m} F_{\alpha,\delta}|U\left(p^m\right)\right)=0.
\end{equation}
\end{enumerate}

\end{proposition}

\begin{proof}

Let a formal power series $H\in \Qp[[q,q^{-1}]]$ be given such that there exists a weakly holomorphic modular form $W\in M_{2-k}^{!}(N,\chi)$ with algebraic coefficients which satisfies
\begin{equation}\label{HBpeqn}
H|B(p) = W.
\end{equation}
Denote 
\begin{equation}\label{H*defeqn}
H^*:=H|\left(1-\beta'p^{1-k} V(p)\right),
\end{equation}
and plug in the definition \eqref{Beqn} of $B(p)$ to equation \eqref{HBpeqn}.  Then acting on both sides of equation \eqref{HBpeqn} by ${\displaystyle \sum_{\ell=0}^{m-1}} \beta^{-\ell}p^{\ell(k-1)} U\left(p^{\ell}\right)$ and taking the $p$-adic limit $m\to\infty$ gives
\begin{equation}\label{H*limeqn}
H^* = \lim_{m\to\infty} \left(\beta^{-m} p^{m(k-1)}H^*|U\left(p^m\right)\right) -\beta^{-1} p^{k-1}\sum_{\ell=0}^{\infty} \beta^{-\ell}p^{\ell(k-1)} W|U\left(p^{\ell}\right),
\end{equation}
with the sum on the right hand side converging by the fact that $W$ has bounded denominators.  Hence $H^*$ is a $p$-adic modular form (of weight $2-k$, level $pN$, and Nebentypus $\chi$, which we suppress hereafter) if and only if 
\begin{equation}\label{H*limpadeqn}
\lim_{m\to\infty} \left(\beta^{-m} p^{m(k-1)}H^*|U\left(p^m\right)\right)
\end{equation}
is a $p$-adic modular form.  We now note that \eqref{H*limpadeqn} is clearly an eigenform under the $U(p)$-operator with eigenvalue $\beta p^{1-k}$.  Hence for any $n\in\N$ and $r\in \N_0$, the $\left(np^r\right)$-th coefficient of \eqref{H*limpadeqn} equals its $n$-th coefficient times $\left( \beta p^{1-k}\right)^r$ and one concludes that the $n$-th coefficient must be zero if \eqref{H*limpadeqn} is a $p$-adic modular form since $p$-adic modular forms have bounded denominators.  Therefore \eqref{H*limpadeqn} is a $p$-adic modular form if and only if it is identically zero.  

Moreover, since $V(p)U(p)$ acts like the identity, we have that
\begin{equation}\label{H*toHeqn}
\lim_{m\to\infty} \left(\beta^{-m} p^{m(k-1)}H^*|U\left(p^m\right)\right) = \left(1-\frac{\beta'}{\beta}\right) \lim_{m\to\infty} \left(\beta^{-m} p^{m(k-1)}H|U\left(p^m\right)\right),
\end{equation}
whenever the right hand side limit exists.  

Next assume that $\widetilde{H}$ is a formal power series such that there exists a weakly holomorphic modular form $\widetilde{W}\in M_{2-k}^!(N,\chi)$ satisfying
$$
\widetilde{H}|B(p)U(p)=\widetilde{W}
$$
and $\ord_p\left(\beta'\right)\neq k-1$.  Defining $\widetilde{H}^*$ analogously to $H^*$ given in \eqref{H*defeqn}, we will now prove that $\widetilde{H}$ is a $p$-adic modular form if and only if 
\begin{equation}\label{H*padiceqn}
\lim_{m\to\infty} \left(\beta^{-m} p^{m(k-1)}\widetilde{H}^*|U\left(p^m\right)\right)
\end{equation}
and 
\begin{equation}\label{Hpadiceqn}
\lim_{m\to\infty}\left({\beta'}^{-m} p^{m(k-1)}\widetilde{H}|U\left(p^m\right)\right)
\end{equation}
are both $p$-adic modular forms.  Plugging in the definition \eqref{Beqn} of $B(p)$, by an argument analogous to that giving equation \eqref{H*limeqn}, one obtains
\begin{equation}\label{Htilde*limeqn}
\widetilde{H}^*|U(p) = \beta p^{1-k} \lim_{m\to\infty} \left(\beta^{-m} p^{m(k-1)}\widetilde{H}^*|U\left(p^{m}\right)\right) -\beta^{-1}p^{k-1} \sum_{\ell=0}^{\infty}\beta^{-\ell}p^{\ell(k-1)} \widetilde{W}|U\left(p^{\ell}\right).
\end{equation}
This shows that $\widetilde{H}^*|U(p)$ is a $p$-adic modular form if and only if \eqref{H*padiceqn} is a $p$-adic modular form (and furthermore identically zero).  Since $\widetilde{H}^*|U(p)$ is clearly a $p$-adic modular form whenever $\widetilde{H}$ is a $p$-adic modular form, one may assume without loss of generality that $\widetilde{H}^*|U(p)$ is a $p$-adic modular form since otherwise both sides of the desired equivalence are false.  Then acting by $U(p)$ on both sides of equation \eqref{H*defeqn} combined with the fact that $\widetilde{H}^*|U(p)$ has bounded denominators, one argues as in the proof of equation \eqref{H*limeqn} to obtain the equality
\begin{equation}\label{HH*evaleqn}
\widetilde{H}=\lim_{m\to\infty}\left({\beta'}^{-m} p^{m(k-1)}\widetilde{H}|U\left(p^m\right)\right)  - {\beta'}^{-1} p^{k-1}\sum_{\ell=0}^{\infty} {\beta'}^{-\ell}p^{\ell(k-1)} \widetilde{H}^*|U\left(p^{\ell+1}\right),
\end{equation}
with the sum converging since $\ord_p\left(\beta'\right)\neq k-1$.  Therefore, it follows that $\widetilde{H}$ is a $p$-adic modular form if and only if \eqref{Hpadiceqn} is a $p$-adic modular form.  By again noting that \eqref{Hpadiceqn} is an eigenform for the $U(p)$-operator with eigenvalue $\beta' p^{1-k}$ and $p$-adic modular forms have bounded denominators, one concludes that $\widetilde{H}$ is a $p$-adic modular form if and only if \eqref{Hpadiceqn} is identically zero.  Thus, we have established that $\widetilde{H}$ is a $p$-adic modular form if and only if \eqref{H*padiceqn} and \eqref{Hpadiceqn} are both identically zero.  However, if \eqref{Hpadiceqn} is identically zero, then it follows that \eqref{H*padiceqn} is identically zero, and hence the statement that $\widetilde{H}$ is a $p$-adic modular form is equivalent to \eqref{Hpadiceqn} equaling zero.

We next establish that the above criterion may be applied to $H=\FF_{\alpha}$ and $\widetilde{H}=\FF_{\alpha,\delta}$.  By equation \eqref{Bprewriteeqn} one sees that equation \eqref{Heckeeqn} for $n=p$ is simply the statement that 
\begin{equation}\label{RpBpeqn}
R_p = M| B(p)=M^+|B(p),
\end{equation}
where we have used the fact that $M^-$ is annihilated by $B(p)=T(p)-a_g(p)p^{1-k}$ (see Lemma 7.4 of \cite{BruinierOnoAnnals}).  But then by Lemma \ref{Bannlemma} and equation \eqref{RpBpeqn} one obtains
\begin{eqnarray*}
\FF_{\alpha}|B(p) &=& M^+|B(p) - \alpha E_g|B(p)=R_p,\\
\FF_{\alpha,\delta}|B(p)U(p)&=&R_p|U(p).
\end{eqnarray*}
Since $\FF_{\alpha}|B(p)$ is a weakly holomorphic modular form, if $\ord_p\left(\beta\right)<\ord_p\left(\beta'\right)$, then $\FF_{\alpha}^*$ is a $p$-adic modular form if and only if 
\begin{equation}\label{HHa*defeqn}
\HH_{\alpha}^*:=\left(\frac{\beta}{\beta-\beta'}\right)\lim_{m\to\infty}\left({\beta}^{-m} p^{m(k-1)}\FF_{\alpha}^* |U\left(p^m\right)\right)=0.
\end{equation}
Moreover, since $\FF_{\alpha,\delta}|B(p)U(p)$ is a weakly holomorphic modular form, equation \eqref{Htilde*limeqn} with $\widetilde{H}=\FF_{\alpha,\delta}$ shows that $\FF_{\alpha,\delta}^*|U(p)$ is a $p$-adic modular form if and only if 
\begin{equation}\label{HHad*defeqn}
\HH_{\alpha,\delta}^{*}:=\lim_{m\to\infty}\left({\beta}^{-m} p^{m(k-1)}\FF_{\alpha,\delta}^* |U\left(p^m\right)\right)=0.
\end{equation}
Furthermore, if $\ord_p\left(\beta'\right)\neq k-1$, then $\FF_{\alpha,\delta}$ is a $p$-adic modular form if and only if 
\begin{equation}\label{HHaddefeqn}
\HH_{\alpha,\delta}:=\lim_{m\to\infty}\left({\beta'}^{-m} p^{m(k-1)}\FF_{\alpha,\delta} |U\left(p^m\right)\right)=0.
\end{equation}
The first and third statements of the proposition now follow by noting that the commutation relation \eqref{UDeqn} between $D^{k-1}$ and the $U(p)$-operator yields $D^{k-1}\left(\HH_{\alpha,\delta}^{*}\right)=h_{\alpha,\delta}^{*}$ and $D^{k-1}\left(\HH_{\alpha,\delta}\right)=h_{\alpha,\delta}$.  The second statement follows by from the fact that $D^{k-1}\left(\HH_{\alpha}^*\right)=h_{\alpha}$ and the existence of the limit on the right hand side of equation \eqref{H*toHeqn}, given by Proposition 2.2 of \cite{PavelKentOno}, in the case when $\ord_p(\beta)<\ord_p(\beta')$.
\end{proof}

Based on Proposition \ref{padiclimprop}, it will be helpful to evaluate $\HH_{\alpha,0}^{*}$, $\HH_{\alpha}^*$, and $\HH_{\alpha,\delta}$.  We first note that for every $\alpha\in\AM$, the limit $\HH_{\alpha,0}^{*}$ exists by equation \eqref{H*limeqn}.  By equation \eqref{Heckeeqn} and the fact that $T(n)$ commutes with $U(p)$ whenever $n$ and $Np$ are relatively prime, one obtains
$$
\HH_{\alpha,0}^{*}|_{2-k}T(n) = n^{1-k} a_g(n) \HH_{\alpha,0}^{*} +\lim_{m\to\infty} \left(\beta^{-m}p^{m(k-1)} R_n^*|U\left(p^m\right)\right),
$$
where $R_n^*:=R_n|\left(1-\beta'p^{1-k} V(p)\right)$.  Moreover, $\lim_{m\to\infty} \left(\beta^{-m}p^{m(k-1)} R_n^*|U\left(p^m\right)\right)=0$ since $R_n^*$ has bounded denominators.  Since $\HH_{\alpha,0}^{*}$ is clearly an eigenform for the $U(p)$-operator with eigenvalue $\beta p^{1-k}$, one may recursively compute the coefficients of $\HH_{\alpha,0}^{*}$ using the $U(p)$-operator and the Hecke operators to establish that
\begin{equation}\label{H*evaleqn}
\HH_{\alpha,0}^{*}=L_{\alpha}^{*}E_{g-\beta'g|V(p)},
\end{equation}
where
\begin{equation}\label{La*defeqn}
L_{\alpha}^{*}:= \lim_{m\to\infty} \left(\beta^{-m}a_{F_{\alpha}^*}\left(p^m\right)\right) \in\Qp.
\end{equation}
Whenever $\ord_p\left(\beta\right)<\ord_p\left(\beta'\right)$, since the limit on the right hand side of equation \eqref{H*toHeqn} exists, one may rewrite equation \eqref{H*evaleqn} as 
\begin{equation}\label{Hevaleqn}
\HH_{\alpha}^*=L_{\alpha}E_{g-\beta'g|V(p)},
\end{equation}
with 
\begin{equation}\label{Ladefeqn}
L_{\alpha}:=\left(\frac{\beta}{\beta-\beta'}\right) L_{\alpha}^{*}= \lim_{m\to\infty} \left(\beta^{-m} a_{F_{\alpha}}\left(p^m\right)\right)\in\Qp.
\end{equation}
Now assume $\ord_p\left(\beta\right)<\ord_p\left(\beta'\right)\neq k-1$.  Note that since $\ord_p\left(\beta\right)<\ord_p\left(\beta'\right)$, if the limit $\HH_{\alpha,\delta}$ exists, then the limit \eqref{H*padiceqn}, with $\widetilde{H}=\FF_{\alpha,\delta}$, must be identically zero.  However, by equation \eqref{Htilde*limeqn}, the limit \eqref{H*padiceqn} exists if and only if $\FF_{\alpha,\delta}^*|U(p)$ is a $p$-adic modular form.  Combining this with equation \eqref{HH*evaleqn} implies that the limit $\HH_{\alpha,\delta}$ exists if and only if $\FF_{\alpha,\delta}^*|U(p)$ is a $p$-adic modular form.  For $(\alpha,\delta)$ given so that $\FF_{\alpha,\delta}^*|U(p)$ is a $p$-adic modular form, one may again use the Hecke operators to obtain that
\begin{equation}\label{Hadevaleqn}
\HH_{\alpha,\delta}=L_{\alpha,\delta} E_{g-\beta g|V(p)},
\end{equation}
where the first coefficient of $\HH_{\alpha,\delta}$ is 
\begin{equation}\label{Laddefeqn}
L_{\alpha,\delta}:=\lim_{m\to\infty} \left({\beta'}^{-m}a_{F_{\alpha,\delta}}\left(p^m\right)\right)\in \Qp.
\end{equation}
We are now ready to treat the generic case.  
\begin{proof}[Proof of Theorem \ref{genericthm}]
By Proposition \ref{padiclimprop} (1), the first part of Theorem \ref{genericthm} (1) is equivalent to showing that precisely one $\alpha\in\AM$ exists so that 
$$
D^{k-1}\left(\HH_{\alpha,0}^{*}\right)=h_{\alpha,0}^{*}=0.
$$
Fix $\alpha_0\in \AM$.  Equation \eqref{H*evaleqn} gives that
\begin{equation}\label{h*alpha0eqn}
h_{\alpha_0,0}^{*}=L_{\alpha_0}^{*} \left(g-\beta' g|V(p)\right),
\end{equation}
where $L_{\alpha_0}^{*}$ is defined in equation \eqref{La*defeqn}.  Let $\gamma\in \Qp$ be given and write $\alpha=\alpha_0+\gamma$.  Then by the definitions \eqref{had*defeqn} of $h_{\alpha,0}^{*}$ and \eqref{F*addefeqn} of $F_{\alpha,0}^{*}$, one has 
$$
h_{\alpha,0}^{*} = h_{\alpha_0,0}^{*} -\gamma \lim_{m\to\infty} \left(\beta^{-m} \left(g-\beta'g|V(p)\right)|U\left(p^{m}\right)\right).
$$
Since $g-\beta'g|V(p)$ is an eigenform under the $U(p)$-operator with eigenvalue $\beta$, equation \eqref{h*alpha0eqn} gives that
$$
h_{\alpha,0}^{*}= \left(L_{\alpha_0}^{*}-\gamma\right) \left(g-\beta' g|V(p)\right).
$$
Thus, $h_{\alpha,0}^{*}=0$ if and only if 
\begin{equation}\label{L0eqn}
\gamma = L_{\alpha_0}^{*}.
\end{equation}
Choosing $\alpha_0=a_M(1)$ above and noting that the first coefficient of $F_{a_M(1)}^*|U\left(p^m\right)$ is precisely $b_M^*\left(p^m\right)$ establishes equation \eqref{alpha*evalgeneqn}.  In the case where $\ord_p\left(\beta\right)<\ord_{p}\left(\beta'\right)$, equations \eqref{Ladefeqn} and \eqref{L0eqn} give that 
$$
\gamma=L_{\alpha_0}^{*}=\frac{\beta-\beta'}{\beta} L_{\alpha}.
$$

Equation \eqref{alphaevalgeneqn} is then established by evaluating $L_{\alpha}$ with the right hand side of equation \eqref{Ladefeqn} combined with the fact that $b_M\left(p^m\right)$ is the first coefficient of $F_{a_M(1)}|U(p^m)$.  This concludes the proof of Theorem \ref{genericthm} (1). 

We now turn to the proof of Theorem \ref{genericthm} (2).  Since $\ord_p(\beta')\neq k-1$, we may use Proposition \ref{padiclimprop} (3), which states that $\FF_{\alpha,\delta}$ is a $p$-adic modular form if and only if $D^{k-1}\left(\HH_{\alpha,\delta}\right)=h_{\alpha,\delta}=0.$  We begin by taking $\alpha_1$ to be the unique choice from above so that $h_{\alpha_1}=0$ and assume that $h_{\alpha,\delta}=0$.  Since $g-\beta g|V(p)$ is an eigenform under the $U(p)$-operator with eigenvalue $\beta'$ and $\ord_p(\beta)<\ord_p(\beta')$, we have that
$$
\lim_{m\to\infty} \left(\beta^{-m} \left(g-\beta g|V(p)\right)|U\left(p^{m}\right)\right)=0.
$$
Therefore, 
$$
\lim_{m\to\infty}\left(\beta^{-m} F_{\alpha,\delta}|U\left(p^m\right)\right)= \lim_{m\to\infty}\left(\beta^{-m} F_{\alpha}|U\left(p^m\right)\right),
$$
and it follows that $\alpha=\alpha_1$.  Indeed, if one has 
$$
\lim_{m\to\infty}\left(\beta^{-m} F_{\alpha,\delta}|U\left(p^m\right)\right)\neq 0,
$$
then the fact that $\ord_p(\beta')>\ord_p(\beta)$ implies that the limit $h_{\alpha,\delta}$ does not even exist.  Because $\FF_{\alpha_1}^*$ is a $p$-adic modular form by Proposition \ref{padiclimprop} (2), equations \eqref{HH*evaleqn} and \eqref{RpBpeqn} with $\widetilde{H}=\FF_{\alpha_1,0}$ imply that the limit $h_{\alpha_1,0}$ exists.  Since $g-\beta g|V(p)$ has eigenvalue $\beta'$ under the $U(p)$-operator, equations \eqref{Hadevaleqn} and \eqref{Laddefeqn} give that
$$
h_{\alpha_1,\delta} = \left(L_{\alpha_1,0}- \delta \right) \left(g-\beta g|V(p)\right),
$$
for $L_{\alpha_1,0}$ defined in equation \eqref{Laddefeqn}.
It again follows that this limit is zero if and only if $\delta = L_{\alpha_1,0}.$
\end{proof}

\section{Shadows with Complex Multiplication}\label{CMSection}

For this section we assume that $g$ has CM.  We will use the following lemma, which follows by a direct calculation and the fact that in this case the coefficients of $M^+$ are contained in $K_g(\zeta)$ for some root of unity $\zeta$ by Theorem 1.3 of \cite{BruinierOnoRhoades}.  For a Dirichlet character $\chi$ and a formal power series $H(q)={\displaystyle \sum_{n\gg -\infty}} a(n) q^n$ we define $H(q)$ \begin{it}twisted by $\chi$\end{it} as 
$$
\left(H\otimes \chi\right)(q) :=\sum_{n\gg -\infty}a(n)\chi(n)q^n.
$$
Furthermore, for a quadratic field $K$ we denote the character associated to $K$ by $\chi_K$.

\begin{lemma}\label{weakhollemma}
Let $M\in H_{2-k}(N,\chi)$ be good for $g$, where $g$ is a cusp form with CM by $K=\Q(\sqrt{-D})$.  Then the twist
$$
R:=\frac{1}{2}\left(M+M\otimes \chi_K\right)\otimes \chi_K
$$
is a weakly holomorphic modular form.  In particular, there exists a weakly holomorphic modular form $R$ with coefficients in $K_g(\zeta)$, where $\zeta$ is a primitive $DN$-th root of unity, so that whenever $\chi_K(n)= 1$ the coefficients $a_M(n)$ and $a_R(n)$ are equal and otherwise we have $a_R(n)=0$.
\end{lemma}

We are now ready to prove Theorem \ref{splitthm}.
\begin{proof}[Proof of Theorem \ref{splitthm}]
Let $p\nmid N$ be split in $\Op$.  By equation \eqref{L0eqn}, the claim that $\alpha=0$ is the unique choice from Theorem \ref{genericthm} (1) is equivalent to showing that $L_{0}=0$.  However, by Lemma \ref{weakhollemma} we know that 
$$
a_{F_0}\left(p^m\right) = a_{D^{k-1}(M)}\left(p^m\right)=a_{D^{k-1}(R)}\left(p^m\right).
$$
Since $R$ is a weakly holomorphic modular form with algebraic coefficients, Proposition 2.1 of \cite{PavelKentOno} implies that there exists an constant $A\geq 0$, depending only on $R$, such that 
$$
\ord_p\left(a_{D^{k-1}(R)}\left(p^m\right)\right)\geq m(k-1)-A,
$$
while $\ord_p(\beta^m)=0$, and it follows that $L_0=0$.
\end{proof}

\begin{proof}[Proof of Proposition \ref{inertprop}]
Let $p\nmid N$ be inert in $\Op$.  By Theorem 1.3 of \cite{BruinierOnoRhoades}, the coefficients of $M^+$ are contained in $K_g(\zeta)$, where $\zeta$ is a primitive $DN$-th root of unity.  But $K_g(\zeta)\subseteq \overline{\Q}\hookrightarrow \Qp$ so that $\widetilde{\FF}_{\alpha}$, as defined in \eqref{FFtadefeqn}, clearly has coefficients in $\Qp$ for every $\alpha\in \Qp$.  We will first show that the $p$-adic limit
\begin{equation}\label{betainerteqn}
W_{\alpha}:= \lim_{m\to\infty} \left(\beta^{-2m}\widetilde{F}_{\alpha}|U\left(p^{2m+1}\right)\right)
\end{equation}
is an element of $\Qp[[q]]$.  By Proposition 2.3 of \cite{PavelKentOno}, limit \eqref{betainerteqn} exists for $\alpha=0$.  Since $g$ is an eigenform under the Hecke operator $T(p)$ with eigenvalue $a_g(p)=0$ and $-\beta^2=\chi(p)p^{k-1}$, one obtains that
\begin{equation}\label{gUpeqn}
g|U(p)=\beta^2 g|V(p).
\end{equation}
Iterating $U(p)$ in equation \eqref{gUpeqn} together with the fact that $g|V(p)|U(p)=g$ gives that
\begin{equation}\label{inertpropevaleqn}
\beta^{-2m} g|V(p)|U\left(p^{2m+1}\right) = g.
\end{equation}
Therefore, the limit $W_{\alpha}$ exists.  Its first coefficient is given by 
$$
\widetilde{L}_{\alpha}:=a_{W_{\alpha}}(1)= \lim_{m\to\infty} \left(\beta^{-2m} a_{\widetilde{F}_{\alpha}}\left(p^{2m+1}\right)\right).
$$
Clearly by evaluating the first coefficient for $W_0$ and determining the linear dependence on $\alpha$ of the first coefficient by equation \eqref{inertpropevaleqn}, one obtains
\begin{equation}\label{Llineqn}
\widetilde{L}_{\alpha}=\widetilde{L}_0 - \alpha.
\end{equation}

Choosing $\delta:=-\frac{\alpha}{\beta}$ and $\widetilde{\alpha}:=-\delta$ gives $F_{\widetilde{\alpha},\delta}=\widetilde{F}_{\alpha}$.  Whenever $(n,pN)=1$, multiplying by $\beta^{-2m}$ and acting by $U\left(p^{2m+1}\right)$ on both sides of equation \eqref{rnadeqn} hence yields the equation
\begin{equation}\label{Ftilderneqn}
\beta^{-2m}\widetilde{F}_{\alpha}|U\left(p^{2m+1}\right) | T(n) = a_g(n)\beta^{-2m} \widetilde{F}_{\alpha}|U\left(p^{2m+1}\right) + \beta^{-2m}r_n|U\left(p^{2m+1}\right).
\end{equation}
Since $r_n\in D^{k-1}\left(M_{2-k}^{!}(N,\chi)\right)$, Proposition 2.1 of \cite{PavelKentOno} implies that 
$$
\lim_{m\to\infty}\left(\beta^{-2m}r_n|U\left(p^{2m+1}\right)\right) =0.
$$
Hence $W_{\alpha}|T(n)=a_g(n)W_{\alpha}$.  Since $W_{\alpha}|U\left(p^2\right)=\beta^2 W_{\alpha}$, inductively computing the coefficients yields that
$$
W_{\alpha}=\widetilde{L}_{\alpha} g +A_p g|V(p),
$$
where $A_p:=a_{W_{\alpha}}(p)$.  To conclude the proposition, it remains to show that $A_p=0$ to establish that  
\begin{equation}\label{Wevaleqn}
W_{\alpha}= \widetilde{L}_{\alpha} g,
\end{equation}
and when $\widetilde{L}_{\alpha}\neq 0$ (namely, precisely when $\alpha\neq \widetilde{L}_0$), the limit in equation \eqref{inertpropeqn} becomes $\widetilde{L}_{\alpha}^{-1} W_{\alpha}=g$.

To show that $A_p=0$, we first note that equation \eqref{inertpropevaleqn} implies that 
\begin{eqnarray*}
A_p&=&\lim_{m\to\infty} \left(\beta^{-2m}a_{D^{k-1}(M)}\left(p^{2m+2}\right)-\alpha \beta^{-2m}a_g\left(p^{2m+1}\right)\right)\\
&=&\lim_{m\to\infty} \left(\beta^{-2m}a_{D^{k-1}(M)}\left(p^{2m+2}\right)-\alpha a_g(p)\right),
\end{eqnarray*}
but $a_g(p)=0$ because $p$ is inert, and hence $A_p$ is independent of $\alpha$.  
Since $\chi_K\left(p^{2m+2}\right)=1$, we know by Lemma \ref{weakhollemma} that 
$$
A_p=\lim_{m\to\infty} \left(\beta^{-2m}a_{D^{k-1}(M)}\left(p^{2m+2}\right)\right)=\lim_{m\to\infty}\left(\beta^{-2m}a_{D^{k-1}(R)}\left(p^{2m+2}\right)\right).
$$
By Proposition 2.1 of \cite{PavelKentOno}, there exists a constant $A\geq 0$ such that $a_{D^{k-1}(R)}\left(p^{2m+2}\right)$ has $p$-order at least $(2m+2)(k-1)-A$, while $\ord_p\left(\beta^{2m}\right) = m(k-1)$, and it follows that $A_p=0$.  This completes the proof of Proposition \ref{inertprop}.

\end{proof}

We now prove Theorem \ref{inertthm}.

\begin{proof}[Proof of Theorem \ref{inertthm}]
Since $g$ has CM by $K$ and $p$ is inert in $\Op$, one has $\beta'=-\beta$.  By Proposition \ref{padiclimprop} (3), $\FF_{\alpha,\delta}$, as defined in \eqref{FFaddefeqn}, is a $p$-adic modular form of weight $2-k$, level $pN$ and Nebentypus $\chi$ if and only if $h_{\alpha,\delta}=0$.  Choose $\delta:=-\frac{\alpha}{\beta}$ and $\widetilde{\alpha}:=-\delta$, so that $\FF_{\widetilde{\alpha},\delta}=\widetilde{\FF}_{\alpha}$.  Hence $h_{\widetilde{\alpha},\delta}=0$ if and only if $W_{\alpha}=0$, where $W_{\alpha}$ was defined in equation \eqref{betainerteqn}.  However, combining equations \eqref{Llineqn} and \eqref{Wevaleqn}, one see that $W_{\alpha}=0$ if and only if 
$$
\alpha=\widetilde{L}_0=\lim_{m\to\infty} \left(\beta^{-2m} a_{\widetilde{F}_{0}}\left(p^{2m+1}\right)\right).
$$
Thus we have shown that $\widetilde{\FF}_{\alpha}$ is a $p$-adic modular form if and only if $\alpha=\widetilde{L}_0$, which is the statement of the theorem.
\end{proof}

We conclude with the case when $p|N$, and in particular the case when $g$ has CM and $p$ is a ramified prime.  We first show that $E_g$ is a $p$-adic modular form.
\begin{proposition}\label{Egclassicalprop}
If $p\mid N$ and $a_g(p)=0$, then $E_g$ is a $p$-adic modular form of level $N$ and Nebentypus $\chi$.
\end{proposition}
\begin{proof}
Since $a_g(p)=0$ and $p\mid N$, one has $a_g(pn)=0$ for every $n\in \N$.  For $n$ relatively prime to $p$, Euler's Theorem states that for any $C\in \N$, we have
$$
n^{C\left(p^{m}-p^{m-1}\right)}\equiv 1 \pmod{p^m}.
$$
Choosing $C$ large enough to satisfy $\ell_m := C(p^m-p^{m-1})-(k-1)>0$ gives, using definition \eqref{Egdefeqn}, that
$$
E_g(z)\equiv \sum_{n\geq 1} a_g(n) n^{C(p^m-p^{m-1})-(k-1)} q^n=D^{\ell_m}(g) \pmod{p^m}.
$$
But the image under $D^{\ell_m}$ of a modular form is itself a $p$-adic modular form (see \cite{Serre72} for the level 1 case), and the proposition follows.
\end{proof}

In order to prove Theorem \ref{divisorsNthm} (1), we will use an argument similar to that given in the proof of Proposition \ref{Egclassicalprop} to show that $\FF_{\alpha}$ is a $p$-adic modular form of level $N$ and Nebentypus $\chi$.
\begin{proof}[Proof of Theorem \ref{divisorsNthm}]
We first assume $a_g(p)=0$ in order to prove Theorem \ref{divisorsNthm} (1).  Since $g|U(p)=0$, one concludes that $g\otimes \chi_p^2=g$.  It follows that $M^+\otimes \chi_p^2$ has shadow $\frac{g}{\| g\|}$ and thus differs from $M^+$ by a weakly holomorphic modular form.  Hence for every $\alpha \in \AM$ one has
\begin{equation}\label{FFdiffeqn}
\FF_{\alpha}\otimes \chi_p^2 - \FF_{\alpha} = M^+\otimes \chi_p^2 - M^+,
\end{equation}
which is a weakly holomorphic modular form.  In particular, for the choice $\alpha=a_M(1)$, the left hand side of equation \eqref{FFdiffeqn} has coefficients in $K_g$ since $\FF_{a_M(1)}$ has its coefficients in $K_g$ by Theorem 1.1 of \cite{PavelKentOno}.  Thus the weakly holomorphic modular form $M^+\otimes \chi_p^2 - M^+$ has coefficients in $K_g$.

Moreover $F_{\alpha}\otimes \chi_p^2 = D^{k-1}\left(\FF_{\alpha}\otimes\chi_p^2\right)$ is a weakly holomorphic modular form.  We then note that for every $n$, the $(pn)$-th coefficient of $\FF_{\alpha}\otimes \chi_p^2$ clearly equals zero.  But then we may argue as in the proof of Proposition \ref{Egclassicalprop}, using Euler's Theorem to approximate $\FF_{\alpha}\otimes \chi_p^2$ by $D^{\ell_m}\left( F_{\alpha}\otimes\chi_p^2 \right)$.  Similarly to the case for holomorphic modular forms, the image of any weakly holomorphic modular form under $D^{\ell_m}$ is a $p$-adic modular form, which can be shown by using the weight 2 Eisenstein series $E_2$ and then noting that $E_2$ is a $p$-adic modular form.  Thus $\FF_{\alpha}\otimes \chi_p^2$ is a $p$-adic modular form which differs from $\FF_{\alpha}$ by a weakly holomorphic modular form with coefficients in $K_g\subseteq\Qp$, and it follows that $\FF_{\alpha}$ is also a $p$-adic modular form.  This concludes the proof of Theorem \ref{divisorsNthm} (1).

To show Theorem \ref{divisorsNthm} (2), we first recall (cf. \cite{OnoBook}, p. 29) that when $a_g(p)\neq 0$ one has 
$$
g|U(p)=a_g(p)g = -\lambda_p p^{\frac{k}{2}-1} g,
$$
where $\lambda_p=\pm 1$ is the eigenvalue of the Atkin-Lehner involution.  Therefore $\ord_p\left(a_g(p)\right)=\frac{k}{2}-1$.  

Therefore, since $M|U(p)$ and $a_g(p) p^{1-k}M$ have the same non-holomorphic part, we have that
\begin{equation}\label{Uopweakeqn}
M^+|U(p) = a_g(p) p^{1-k}M^+ +R_p
\end{equation}
for some $R_p\in M_{2-k}^{!}(N,\chi)$.  Rewriting equation \eqref{Uopweakeqn} and noting that 
$$
E_g|\left(1-a_g(p)^{-1}p^{k-1} U(p)\right)=0,
$$
one obtains that
\begin{equation}\label{RpUeqn}
R_p =  -a_g(p)p^{1-k} \FF_{\alpha}|\left(1-a_g(p)^{-1}p^{k-1} U(p)\right).
\end{equation}
We then act on both sides of equation \eqref{RpUeqn} by $\sum_{\ell=0}^{m-1} \frac{p^{\ell(k-1)}}{a_g(p)^{\ell}} U\left(p^{\ell}\right)$ and take the limit $m\to\infty$.  Since $R_p$ has bounded denominators and $\ord_p\left(a_g(p)\right)=\frac{k}{2}-1<k-1$, one obtains that 
\begin{equation}\label{pdivNMevaleqn}
\FF_{\alpha} = \lim_{m\to\infty}\left( a_g(p)^{-m} p^{m(k-1)}\FF_{\alpha}|U\left(p^m\right)\right)- a_g(p)^{-1}p^{k-1} \sum_{\ell=0}^{\infty}a_g(p)^{-\ell} p^{\ell(k-1)} R_p|U\left(p^{\ell}\right).
\end{equation}
It follows that $\FF_{\alpha}$ is a $p$-adic modular form of weight $2-k$, level $N$, and Nebentypus $\chi$ if and only if 
\begin{equation}\label{pdivNlimeqn}
G_{\alpha}:=\lim_{m\to\infty}\left( a_g(p)^{-m}p^{m(k-1)} \FF_{\alpha}|U\left(p^m\right)\right)
\end{equation}
is a $p$-adic modular form.  
However, the limit \eqref{pdivNlimeqn} exists by equation \eqref{pdivNMevaleqn} and, analogously to the proof of equation \eqref{H*evaleqn}, one can check easily that $G_{\alpha}$ is a Hecke eigenform for all of the Hecke operators, giving, by recursively computing coefficients, that 
$$
G_{\alpha} = \Lambda_{\alpha} E_g,
$$
where the first coefficient of $G_{\alpha}$ is given by 
\begin{equation}\label{LUdefeqn}
\Lambda_{\alpha}:= \lim_{m\to\infty}\left( a_g(p)^{-m} a_{F_{\alpha}}\left(p^m\right)\right).
\end{equation}
We fix $\alpha_0\in \AM$ and note that
$$
\lim_{m\to\infty}\left( a_g(p)^{-m}p^{m(k-1)} E_g|U\left(p^m\right)\right)=E_g,
$$
so that for $\alpha=\alpha_0+\gamma$ with $\gamma\in \Qp$ one obtains that 
$$
G_{\alpha} = \left( \Lambda_{\alpha_0} - \gamma\right) E_g.
$$
Since $E_g$ is not a $p$-adic modular form by Lemma \eqref{uniquelemma} (2), $G_{\alpha}$ is a $p$-adic modular form if and only if 
$$
\gamma = \Lambda_{\alpha_0}.
$$
Plugging in $\alpha_0=a_M(1)$ gives the value for $\alpha$ given in equation \eqref{alphadivNeqn}.  This concludes the proof of the theorem.
\end{proof}

\end{document}